\documentclass[12pt]{article}
\usepackage{amsthm, amsfonts}
\usepackage{url}
\usepackage{amscd,amsmath,amssymb}

\newtheorem{theorem}{Theorem}
\newtheorem{lemma}{Lemma}

\newtheorem{corollary}{Corollary}

\def\de{\delta}

\def\al{\alpha}
\def\Ga{\Gamma}

\def\La{\Lambda}
\def\la{\lambda}
\def\Om{\Omega}
\def\om{\omega}

\def\kappa{\varkappa}

\def\C{{\mathbb C}}

\def\Z{{\mathbb Z}}

\def\sN{{\mathfrak S}_N}
\def\sl{\mathfrak{sl}_2}
\def\sinf{{\mathfrak S}_{\mathbb N}}
\def\maj{\operatorname{maj}}
\def\des{\operatorname{des}}
\def\Vir{\operatorname{Vir}}
\def\te{{\widetilde e}}

\def\T{{\cal T}}
\def\H{{\cal H}}
\def\X{{\cal X}}
\def\F{{\cal F}}
\def\K{{\cal K}}

\def\beq{\begin{equation}}
\def\eeq{\end{equation}}

\urldef\vershik\url{vershik@pdmi.ras.ru}
\urldef\natalia\url{natalia@pdmi.ras.ru}

\author{N.~V.~Tsilevich\thanks{%
St.~Petersburg Department of Steklov Institute of Mathematics and St.~Petersburg State University.
E-mail: \natalia, \vershik. Supported by the grants
RFBR 13-01-12422-ofi-m and RFBR 14-01-00373-a.}
\and A.~M.~Vershik\footnotemark[1]}
\title{On a relation between the basic representation of the affine Lie algebra $\widehat\sl$ and a Schur--Weyl representation of the infinite symmetric group}

\begin{document}
\maketitle

\begin{abstract}
We prove that there is a natural grading-preserving isomorphism of $\sl$-modules between
the basic module of the affine Lie algebra $\widehat\sl$ (with the homogeneous grading) and a Schur--Weyl module of the infinite symmetric group $\sinf$ with a grading defined through the combinatorial
notion of the major index of a Young tableau, and study the properties of this isomorphism. The results
reveal new and deep interrelations between the representation theory of $\widehat\sl$ and the Virasoro algebra on the one hand, and the representation theory of $\sinf$ and the related combinatorics on the other hand.
\end{abstract}

\section{Introduction}

In this paper we reveal some new and deep interrelations between two well developed branches
of representation theory: the representation theory of the infinite symmetric group and that of the
affine Lie and Virasoro algebras.
Our starting point  was the analogy observed in
\cite{SW} between the decomposition~\eqref{X} of a so-called Schur--Weyl representation of the infinite symmetric group $\sinf$ into irreducibles,
$$
\X=\sum_{k=0}^\infty M_{2k+1}\otimes\Pi_{k},
$$
and the decomposition~\eqref{virdecomp} of the basic representation of the affine Lie algebra $\widehat\sl$ into irreducible representations of the Virasoro algebra $\Vir$,
$$
\H_0=\bigoplus_{k=0}^\infty M_{2k+1}\otimes L(1,k^2);
$$
in these formulas $\Pi_{k}$ is an irreducible representation of $\sinf$, $L(1,k^{2})$ is an irreducible representation of $\Vir$, and $M_{2k+1}$ is the $(2k+1)$-dimensional irreducible representation of $\sl$; in both cases,  the operator algebras generated by the
actions of $\sinf$ or $\Vir$ and $\sl$ are mutual commutants.
This analogy suggested that there should be a natural action of $\Vir$ in the $\sinf$-module $\X$, or, equivalently, a natural action of $\sinf$ in the
$\widehat\sl$-module $\H_{0}$. The aim of this paper is to describe and study the underlying natural isomorphism of $\sl$-modules.

For this, we use the result of B.~Feigin and E.~Feigin \cite{FF} that the level~1 irreducible
highest weight representations of $\widehat\sl$ can be realized as certain inductive limits of tensor powers
$(\C^{2})^{\otimes N}$ of the two-dimensional irreducible representation of $\sl$. The construction of \cite{FF} is based on the notion of the fusion product of representations, whose
main ingredient is, in turn, a special grading in the space $(\C^{2})^{\otimes N}$. A key observation underlying the results of this paper is that the fusion product under consideration can be realized in
an $\sN$-module so that this special grading essentially coincides with a well-known
combinatorial characteristic of Young tableaux called the major index (see Sec.~\ref{sec:major}
and Theorem~\ref{th1}). Thus
our results provide, in particular, a kind of combinatorial description of the fusion product and show that
the combinatorial notion of the major index of a Young tableau has new and rich representation-theoretic meaning. For instance, Corollary~\ref{cor:vir} in Sec.~\ref{sec:iso}  shows that the so-called stable major indices of infinite Young tableaux are the eigenvalues of the Virasoro $L_{0}$ operator, the Gelfand--Tsetlin basis of the Schur--Weyl module being its eigenbasis.

The paper is organized as follows. In Secs.~\ref{sec:sw} and \ref{sec:fusion} we briefly reproduce the
necessary background on the notion of Schur--Weyl duality and the fusion product of representations,
respectively.  Section~\ref{sec:major} contains our finite-dimensional Theorem~\ref{th1} with
combinatorial interpretation of the fusion product grading via the major index of Young tableaux.
In Sec.~\ref{sec:limit}, we prove its infinite-dimensional version, our main Theorem~\ref{th2},
which states that the grading-preserving isomorphism of $\sl$-modules constructed in Theorem~
\ref{th1} extends, through the corresponding inductive limits, to
a grading-preserving isomorphism of $\sl$-modules between the basic $\widehat\sl$-module
$L_{0,1}$ and the Schur--Weyl module $\X$. The
remaining part of the paper is devoted to studying the key isomorphism in more detail. With this aim,
in Sec.~\ref{sec:fock} we describe the Fock space realizations of the involved representations
of $\widehat\sl$ and $\Vir$, and then, in Sec.~\ref{sec:iso}, prove some properties of our
isomorphism (Theorem~\ref{th3}).

\medskip
{\it For definiteness, in what follows we consider only the even case
$N=2n$. The odd case can be treated in exactly the same way; instead
of the basic representation $L_{0,1}$, it leads to the other level~$1$ highest weight representation
$L_{1,1}$ of $\widehat\sl$.}
\bigskip

\noindent {\bf Acknowledgments.} The authors are grateful to Igor Frenkel for many inspiring discussions, and also to Boris Feigin and Evgeny Feigin for introducing into the notion of fusion product.

\section{Infinite-dimensional Schur--Weyl duality}\label{sec:sw}
In \cite{SW}, the notion of infinite-dimensional Schur--Weyl duality
was introduced. Namely, starting from the classical Schur--Weyl
duality
\beq\label{swfin}
(\C^2)^{\otimes N}
=\sum_{k=0}^{n}M_{2k+1}\otimes\pi_{k},
\eeq
where $\pi_{k}$ is the irreducible representation of the
symmetric group $\sN$
corresponding to the two-row Young diagram
$\la^{(k)}=(n+k,n-k)$
and $M_{2k+1}$ is the $(2k+1)$-dimensional irreducible
representation of the special linear group $SL(2,\C)$,
we consider so-called Schur--Weyl embeddings $(\C^2)^{\otimes
N}\hookrightarrow(\C^2)^{\otimes (N+2)}$ that preserve this
Schur--Weyl structure, i.e.,
respect both the actions of $SL(2,\C)$ and $\sN$, and the
inductive limits of chains
\beq\label{even}
(\C^2)^{\otimes0}\hookrightarrow(\C^2)^{\otimes
2}\hookrightarrow(\C^2)^{\otimes
4}\hookrightarrow{\ldots}.
\eeq
Such an inductive limit has the form
\beq\label{X}
\X=\sum_{k=0}^\infty M_{2k+1}\otimes\Pi_{k},
\eeq
where $\Pi_{k}$ is an irreducible representation of the infinite
symmetric group $\sinf$
(an inductive limit of the sequence of irreducible
representations $\pi_k$ of $\sN$);
the operator algebras generated by the
actions
of $\sinf$ and $SL(2,\C)$ are mutual commutants.

\section{Fusion product}\label{sec:fusion}

The notion of the fusion product of finite-dimensional representations
of $\sl$ was introduced in \cite{FL}. Given an $\sl$-representation $\rho$ and
$z\in\C$, let $\rho(z)$ be the evaluation representation of the polynomial
current algebra $\sl\otimes\C[t]$, defined as $(x\otimes
t^i)v=z^i\cdot xv$ for $x\in\sl$, $v\in\rho$. Now, given a collection
$\rho_1,{\ldots} ,\rho_N$ of irreducible representations of $\sl$
with lowest weight vectors $v_1,{\ldots} ,v_N$, and
a collection $z_1,{\ldots} ,z_N$ of pairwise distinct complex numbers,
we consider the tensor product of the corresponding evaluation
representations: $ V_N=\rho_1(z_1)\otimes{\ldots} \otimes\rho_n(z_N)$.
The crucial step is introducing a special grading in $V_N$ by setting
$$
V_N^{(m)}=U^{(m)}(e\otimes\C[t])(v_1\otimes{\ldots} \otimes
v_n)\subset V_N,
$$
where $e=\begin{pmatrix} 0&1\\0&0\end{pmatrix}$ is the raising
operator in $\sl$ and
$U^{(m)}$ is spanned by homogeneous elements of degree $m$ in $t$. In
other words, $V_N^{(m)}$  is spanned by the monomials of the form
$$
e_{i_1}{\ldots} e_{i_k},\quad i_1+{\ldots} +i_k=m,
$$
where $e_j=e\otimes t^j$. Then we consider the
corresponding filtration on $V_N$:
$$
V_N^{(\le m)}=\bigoplus_{k\le m} V_N^{(k)}.
$$

The fusion product of $\rho_1,{\ldots} ,\rho_N$ is the graded
representation with respect to the above filtration:
\beq\label{gr}
V_N^*=\operatorname{gr}V_N=V_N^{(\le0)}\oplus V_N^{(\le1)}/V_N^{(\le0)}\oplus
V_N^{(\le2)}/V_N^{(\le1)}\oplus{\ldots} .
\eeq
The space $V_N^*[k]=V_N^{(\le k)}/V_N^{(\le k-1)}$ is the subspace of
elements of degree $k$, and elements of the form $x\otimes
t^l\in\sl\otimes C[t]$ send $V_N^*[k]$ to $V_N^*[k+l]$.
The degree of an element with respect to this grading will be denoted
by $\widetilde\deg$.

It is proved in \cite{FL} that $V_N^*$ is an
$\sl\otimes(\C[t]/t^n)$-module that does not depend on $z_1,{\ldots}
,z_N$ provided that they are pairwise distinct. Moreover, $V_N^*$ is
isomorphic to $\rho_1\otimes{\ldots} \otimes\rho_N$ as an $\sl$-module.

We apply this construction to the case where
$\rho_1={\ldots} =\rho_N=M_2$ with $M_2=\C^2$ being the two-dimensional
irreducible representation of $\sl$ with the lowest weight
vector $v_0$.
In this case,
$$
V_N^*\simeq(\C^2)^{\otimes N}\quad\mbox{as an
$\sl$-module}.
$$
We equip $V_N^*$ with the inner product such that the corresponding representation
of $\sl$ is unitary.

Consider the decomposition of $V_N^*$ into irreducible $\sl$-modules:
$$
V_N^*=\bigoplus_{k=0}^nM_{2k+1}\otimes{\cal M}_k.
$$
By the
classical Schur--Weyl duality~\eqref{swfin}, we know that the
multiplicity space ${\cal M}_k$ is the space of the irreducible
representation $\pi_k$ of $\sN$.
On the other hand, it
inherits the grading from $V_N^*$:
\beq\label{mult}
{\cal M}_k=\bigoplus_{i\ge0}{\cal M}_k[i],
\eeq
where ${\cal M}_k[i]={\cal M}_k\cap V^*[i]$.
Consider the corresponding $q$-character
$$
\operatorname{ch}_q{\cal M}_k=\sum_{i\ge0}q^i\dim{\cal M}_k[i].
$$
It was proved in \cite{Kedem} that
\beq\label{kedem}
\operatorname{ch}_q{\cal M}_k=q^{\frac{N(N-1)}2}\cdot K_{\la_k,1^N}(1/q),
\eeq
where $K_{\la,\mu}$ is the Kostka--Foulkes polynomial (see
\cite[Sec.~III.6]{Mac}).

\section{Major index and the tableaux realization of the fusion product}\label{sec:major}

Let $T_N$ be the set of all standard Young tableaux of length $N$
with at most two rows.

As was proved in \cite{LS},
\beq\label{LS}
K_{\la,1^N}(q)=\sum_{\tau\in[\la]}q^{c(\tau)},
\eeq
where  $[\la]$ is the set of standard Young tableaux of shape $\la$
and $c(\tau)$ is the so-called charge of a tableau $\tau\in T_N$, defined
as the sum of $i\le N-1$ such that in $\tau$ the element $i+1$ lies to
the right of $i$ (see \cite{Mac}).

It is more convenient for our purposes to use another statistic on
Young tableaux, namely,
the major index, defined as
follows (see \cite[Sec.~7.19]{St}):
$$
\maj(\tau)=\sum_{i\in\des(\tau)} i,
$$
where, for $\tau\in T_{N}$,
$$
\des(\tau)=\{i\le N-1: \mbox{ the element $i+1$  in $\tau$ lies lower than $i$}\}
$$
is the descent set of $\tau$. Obviously, for $\tau\in T_N$ we have
$\maj(\tau)=\frac{N(N-1)}2-c(\tau)$.
Then it follows from
\eqref{kedem} and \eqref{LS} that
\beq\label{dim}
\dim{\cal M}_k[i]=\#\{\tau\in[(n+k,n-k)]:\maj(\tau)=i\}.
\eeq

Denote by $\X_N$ the space $(\C^2)^{\otimes N}
=\sum_{k=0}^{n}M_{2k+1}\otimes\pi_{k}$ (see~\eqref{swfin}) in which
the irreducible representation $\pi_k$ of $\sN$ is realized in
the space spanned by the standard Young tableaux of shape $(n+k,n-k)$
equipped with the standard inner product under which the
representation is unitary. Note that this is an $\sl$-module endowed additionally with the grading $\maj$.

\begin{theorem}\label{th1}
There is a grading-preserving unitary isomorphism of the fusion product $V_N^*$ (with the grading $\widetilde\deg$) and the
space $\X_N$ (with the grading $\maj$) as $\sl$-modules such that
the multiplicity space ${\cal M}_k$ is spanned by the standard Young tableaux $\tau$ of shape
$(n+k,n-k)$ (and hence  ${\cal M}_k[i]$ is spanned by $\tau$ with $\maj(\tau)=i$).
\end{theorem}
\begin{proof}
Follows from the fact that the fusion product
$V_N^*$ is isomorphic to $(\C^2)^{\otimes N}$ as an $\sl$-module and
equation~\eqref{dim}.
\end{proof}

\smallskip\noindent{\bf Remark 1.}
Observe that the isomorphism from Theorem~\ref{th1} is not
unique.
\medskip

\noindent{\bf Remark 2.} The isomorphism from Theorem~\ref{th1} determines
an action of the symmetric group $\sN$ on the space $V_N^*$. It does
not coincide with the original action of $\sN$ on $\C^{\otimes N}$.
\medskip

Given $\tau\in T_N$, let $k(\tau)$ be half the difference  of the lengths of the
first and the second row of $\tau$. Then, in view of the Schur--Weyl
duality, we can write
$$
V_N^*=\bigoplus_{\tau\in T_N}M_{2k(\tau)+1}(\tau),
$$
where $M_{2k(\tau)+1}(\tau)$ is the $(2k(\tau)+1)$-dimensional
$\sl$-module parametrized by $\tau$ as an element of the multiplicity
space.

\section{Embeddings and the limit}\label{sec:limit}

It is proved in \cite{FF} that there is an embedding
$$
j_N:V_N^*\to V^*_{N+2}
$$
equivariant with respect to the action of $\sl\otimes(\C[t^{-1}]/t^{-n})$, and
the corresponding inductive limit
$$
{\cal V}=\lim(V_N,j_N)
$$
is isomorphic to the basic representation $L_{0,1}$ of
the affine Lie algebra $\widehat\sl$.
This embedding satisfies
\beq\label{embdeg}
\widetilde\deg(j_Nx)=\widetilde\deg(x)-(N+1).
\eeq

Now consider the following natural embedding $i_N:T_N\to
T_{N+2}$: given a standard Young tableau $\tau$ of length $N$, its
image $i_N(\tau)$ is the standard Young tableau of length $N+2$
obtained from $\tau$  by adding the element $N+1$ to the first row and
the element $N+2$ to the second row.

Note that $i_N$ is, obviously, a Schur--Weyl
embedding in the sense of \cite{SW} (see Sec.~\ref{sec:sw}). Let
$\X$ be the corresponding inductive limit~\eqref{X}. Then $\Pi_k$ is
the discrete representation of the infinite symmetric group $\sinf$
associated with the tableau
\beq\label{tauk}
\tau_k=\begin{array}{ccccccc}
1&2&{\ldots} &2k&2k+1&2k+3&{\ldots} \\
2k+2&2k+4&{\ldots} &&&&
\end{array},
\eeq
which can be realized in the space (which, by abuse of notation, will also be denoted by
$\Pi_k$) spanned by the infinite two-row Young tableaux
tail-equivalent to $\tau_k$ (we denote the set of such tableaux by
$\T_k$). In what follows,
the tableaux $\tau_k$ will be called {\it principal}.

Obviously,
\beq\label{embmaj}
\maj(i_N(\tau))=\maj(\tau)+(N+1).
\eeq
Given $N=2n$ and $\tau\in T_N$,
denote $r_N(\tau)=n^2-\maj(\tau)$. Then
$r_{N+2}(i_N(\tau))=r_N(\tau)$, so that we have a well-defined grading
on the space $\Pi=\bigoplus_{k=0}^\infty\Pi_k$:
\beq\label{r}
r(\tau)=\lim_{n\to\infty}r_{2n}([\tau]_{2n})=\lim_{n\to\infty}(n^2-\maj([\tau]_{2n})),
\eeq
where $[\tau]_l$ is the initial part of length $l$ of the infinite
tableau $\tau$. We will call $r(\tau)$ the {\it  stable major index} of $\tau$.
Obviously, $r(\tau_k)=k^2$.

Our main theorem is the following.

\begin{theorem}\label{th2}
The grading-preserving unitary isomorphism of $\sl$-modules described in Theorem~{\rm\ref{th1}} extends to
a grading-preserving unitary isomorphism of $\sl$-modules between the spaces $\cal V$
and $\X$:
\beq\label{fpdecomp}
{\cal V}\simeq\X=\sum_{k=0}^\infty M_{2k+1}\otimes\Pi_{k}.
\eeq
Thus in the Schur--Weyl module $\X$, which is an $\sl$-module
and an $\sinf$-module, there is also a structure of
the basic $\widehat\sl$-module $L_{0,1}$. The corresponding
grading is given by the stable major index~\eqref{r}, that is,
for $w=x\otimes\tau\in M_{2k+1}\otimes\Pi_k$, we have
$\deg w=r(\tau)$.
\end{theorem}

\smallskip\noindent{\bf Remark.}
As mentioned in the introduction, we consider in detail only the even case
just for simplicity of notation. Considering instead of \eqref{even} the chain
$(\C^2)^{\otimes1}\hookrightarrow(\C^2)^{\otimes
3}\hookrightarrow(\C^2)^{\otimes
5}\hookrightarrow{\ldots}$ and reproducing exactly
the same arguments, we will obtain a grading-preserving isomorphism of the corresponding Schur--Weyl representation with the other level~$1$ highest weight representation
$L_{1,1}$ of $\widehat\sl$.
\medskip

\begin{proof}
Since we are now considering
$\sl\otimes\C[t^{-1}]$ instead of $\sl\otimes\C[t]$, we should
slightly modify the previous constructions to take the minus sign into
account. Namely,
instead of~\eqref{mult} we now have
${\cal M}_{k}=\oplus_{i\ge0}{\cal M}_{k}[-i]$, and the isomorphism of
Theorem~\ref{th1} identifies ${\cal M}_{k}[-i]$ with the space spanned
by the tableaux $\tau$ of
shape $(n+k,n-k)$ such that $\maj(\tau)=i$. Denote this isomoprhism
between $V^*_N$ and $\X_N$ by $\rho_N$. Observe that the only
conditions we impose on $\rho_N$ are as follows: (a) $\rho_N$ is a
unitary isomorphism
of $\sl$-modules and (b) $\rho_N\circ\widetilde\deg=-\maj$.

Now, to prove Theorem~\ref{th2}, we need to show that we can choose
a sequence of isomorphisms $\rho_N$ such that the diagram
$$
\begin{CD}
V_N^*@>\rho_N>>X_N\\
@VVj_NV @ VVi_NV\\
V_{N+2}^*@>\rho_{N+2}>>X_{N+2}
\end{CD}
$$
is commutative for all $N$. We use induction on $N$. The base being
obvious, assume that we have already constructed $\rho_N$, and let us
construct $\rho_{N+2}$.

We have $V_{N+2}^*=j_N(V_N^*)\oplus (j_N(V_N^*))^\perp $. On the first subspace, we
set $\rho_{N+2}(x):=i_N(\rho_N(j_{N+2}^{-1}(x)))$. On the second
one, we define it in an arbitrary way to satisfy the desired
conditions~(a) and~(b). The fact that this definition is correct and
provides us with a desired isomorphism between $V_{N+2}^*$ and
$\X_{N+2}$ follows
from~\eqref{embdeg} and~\eqref{embmaj}.
\end{proof}

\begin{corollary}
The embedding $j_N:V_N^*\to V_{N+2}^*$ is equivariant with respect to
the action of the symmetric group ${\frak S}_N$ (see Remark~{\rm 2}
after Theorem~{\rm\ref{th1}}). Thus the limit space $\cal V$,
isomorphic to $L_{0,1}$, has the structure of a representation of the
infinite symmetric group $\sinf$.
\end{corollary}

Let $\om_{-2k}$ be the lowest vector in $M_{2k+1}$. Then
a natural basis of $\cal V$ is
$
\{e_0^m\om_{-2k}\otimes\tau: m=0,1,{\ldots} ,2k,\;\tau\in\T_k\}
$.
Denoting ${\cal V}_k=M_{2k+1}\otimes\Pi_k$ and ${\cal
V}_k[0]=\{v\in{\cal V}_k:h_0v=0\}$, we have ${\cal
V}_k[0]=e_0^k\om_{-2k}\otimes\Pi_k$, so that
we may identify
${\cal V}_k[0]$ with $\Pi_k$ by the correspondence
$$
c(t)\cdot e_0^k\om_{-2k}\otimes t\leftrightarrow t,\qquad t\in\Pi_k,
$$
where $c(t)$ is a normalizing constant. Thus we have
\beq\label{corr0}
{\cal V}[0]:=\{v\in{\cal V}:h_0v=0\}\longleftrightarrow\Pi=\bigoplus_{k=0}^\infty\Pi_k,
\eeq
where $\Pi$ is the space spanned by all infinite two-row Young tableaux with
``correct'' tail behavior, i.e., tail-equivalent to $\tau_k$
(see~\eqref{tauk}) for some $k$.

Our aim in the remaining part of the paper is to study the isomorphism from Theorem~\ref{th2} in more detail. For this, we first describe the Fock space realization of the basic $\widehat\sl$-module and the fusion product.

\section{The Fock space}\label{sec:fock}

\subsection{The Fock space and the level~1 highest weight representations of
$\widehat\sl$}
Let $\F$ be the fermionic Fock space constructed as the infinite wedge
space over the linear space with basis
$\{u_k\}_{k\in\Z}\cup\{v_k\}_{k\in\Z}$.
That is, $\F$ is spanned by the semi-infinite forms
\begin{eqnarray*}
u_{i_1}\wedge{\ldots}\wedge u_{i_k}\wedge v_{j_1}\wedge{\ldots}\wedge v_{j_l}\wedge
u_N\wedge v_N\wedge u_{N-1}\wedge v_{N-1}\wedge{\ldots},\\
N\in\Z,\; i_1>{\ldots} >i_k>N,\; j_1>{\ldots} >j_l>N,
\end{eqnarray*}
and is equipped with the inner product in which such monomials are
orthonormal.
Let $\phi_k$ be the exterior multiplication by $u_k$
and $\psi_k$ be the exterior multiplication by $v_k$,
and denote by $\phi_k^*$, $\psi_k^*$ the corresponding adjoint
operators. Then this family of operators satisfies the canonical
anticommutation relations (CAR):
$$
\phi_k\phi_k^*+\phi^*_k\phi_k=1,\qquad \psi_k\psi_k^*+\psi^*_k\psi_k=1,
$$
all the other anticommutators being zero.

Consider the generating functions
$$
\phi(z)=\sum_{i\in\Z}\phi_iz^{-(i+1)},\quad\psi(z)=\sum_{i\in\Z}\psi_iz^{-(i+1)},\quad
\phi^*(z)=\sum_{i\in\Z}\phi_i^*z^i,\quad\psi^*(z)=\sum_{i\in\Z}\psi_i^*z^i.
$$

Let $a^\phi_n$ and $a^\psi_n$
be the systems of bosons constructed from the fermions $\{\phi_k\}$ and
$\{\psi_k\}$, respectively:
$$
a^\phi_0=\sum_{n=1}^\infty\phi_{n}\phi_{n}^*-\sum_{n=0}^\infty\phi_{-n}^*\phi_{-n},\quad
a^\phi_n=\sum_{k\in\Z}\phi_{k}\phi^*_{k+n},\quad n\ne0,
$$
and similarly for $a^\psi$. They satisfy the canonical commutation
relations (CCR)
\beq\label{CCR}
[a_n^\phi,a_m^\phi]=n\de_{n,-m},\qquad [a_n^\psi,a_m^\psi]=n\de_{n,-m},
\eeq
i.e., form a representation of the Heisenberg algebra $\mathfrak A$.
Denote
$$
a^\phi(z)=\sum_{n\in\Z}a^\phi_nz^{-(n+1)},\quad
a^\psi(z)=\sum_{n\in\Z}a^\psi_nz^{-(n+1)}.
$$

Let $V$ be the operator in $\F$ that shifts the indices by 1:
$$
V(w_{i_1}\wedge w_{i_2}\wedge{\ldots} )=V_0(w_{i_1})\wedge
V_0(w_{i_2})\wedge{\ldots},\quad V_0(u_i)=u_{i+1},\quad
V_0(v_i)=v_{i-1}.
$$

The vacuum vector in $\F$ is $\Om=u_{-1}\wedge v_{-1}\wedge
u_{-2}\wedge v_{-2}\wedge{\ldots} $. We also consider the family of
vectors
$$
\Om_0=\Om,\quad \Om_{2n}=V^{-n}\Om_0,\quad n\in\Z.
$$

In the space $\F$ we have a canonical representation of the affine
Lie algebra $\widehat\sl=\sl\otimes\C[t,t^{-1}]\oplus\C c\oplus\C d$, which is given by the following formulas.
Given $x\in\sl$, denote $X(z)=\sum_{n\in\Z}x_nz^{-(n+1)}$. Then
\begin{eqnarray*}
E(z)=\psi(z)\phi^*(z),\qquad
F(z)=\phi(z)\psi^*(z),\\
h_n=a^\psi_{-n}-a^\phi_{-n},\qquad
d=\frac{h_0^2}2+\sum_{n=1}^\infty h_{-n}h_n,\qquad c=1. 
\end{eqnarray*}

We have
$$
\F=\H_0\otimes\K_0+\H_1\otimes\K_1,
$$
where $\H_0\simeq L_{0,1}$ and $\H_1\simeq L_{1,1}$ are the
irreducible  level~1 highest weight representations of
$\widehat\sl$ and  $\K_0$ and $\K_1$ are the multiplicity spaces.
Observe also that
$$
e_{-(N+1)}\Om_{-N}=\Om_{-(N+2)}.
$$

Note that the operators $a_n=\frac1{\sqrt{2}}h_n$
satisfy the CCR~\eqref{CCR}, i.e., form a system of free bosons, or
generate the Heisenberg algebra ${\mathfrak A}_h$.
The vectors $\{\Om_{2n}\}_{n\in\Z}$ introduced above are exactly singular vectors for
this Heisenberg algebra:
$h_k\Om_m=0$ for $m<0$, $h_0\Om_m=m\Om_m$. The representation of
${\mathfrak A}_h$ in $\H_0$ breaks into a direct sum of
irreducible representations:
\beq
\H_0=\bigoplus_{k\in\Z}\H_0[2k],
\eeq
where $\H_0[2k]$ is the charge~$2k$ subspace, i.e., the eigenspace of
$h_0$ with eigenvalue $2k$:
$$\H_0[2k]=\{v\in\H_0:h_0v=2kv\}=\C[h_0,h_1,{\ldots} ]\Om_{2k}.
$$

\subsection{The representation of the Virasoro algebra
associated with the basic representation of $\widehat\sl$}

Given a representation of the affine Lie algebra $\widehat\sl$, we can
use the Sugawara construction to obtain the corresponding
representation of the Virasoro algebra $\Vir$. It can also be
described in the following way.  As noted above, the operators $a_n=\frac1{\sqrt{2}}h_n$
form a system of free bosons. Given
such a system, a representation of $\Vir$ can be constructed as
follows (\cite[Ex.~9.17]{Kac}):
$$
L_n=\frac12\sum_{j\in\Z}a_{-j}a_{j+n},\quad n\ne0;\qquad
L_0=\sum_{j=1}^\infty a_{-j}a_j.
$$
Thus we obtain a representation of $\Vir$ in $\F$ and, in particular,
in $\H_0$. In this representation, the algebras generated by the
operators of $\Vir$ and
$\sl\subset\widehat\sl$ are mutual commutants, and we have the decomposition
\beq\label{virdecomp}
\H_0=\bigoplus_{k=0}^\infty M_{2k+1}\otimes L(1,k^2),
\eeq
where $M_{2k+1}$ is the $(2k+1)$-dimensional irreducible representation of $\sl$
and $L(1,k^2)$ is the irreducible representation of $\Vir$ with
central charge~1 and conformal dimension $k^2$.

The charge $k$ subspace $\H_0[k]$ contains a series of singular
vectors $\xi_{k,m}$ of $\Vir$  with energy $(k+m)^2$:
$$
L_n\xi_{k,m}=0 \mbox{ for }n=1,2,{\ldots} ,\qquad L_0\xi_{k,m}=(k+m)^2.
$$

Let us use the so-called
homogeneous vertex operator construction of the basic
representation of $\widehat\sl$ (see \cite[Sec.~14.8]{Kac}). In this realization,
\beq\label{EF}
E(z)=\Ga_-(z)\Ga_+(z)z^{-h_0}V^{-1},\qquad
F(z)=\Ga_+(z)\Ga_-(z)z^{h_0}V,
\eeq
where
$$
\Ga_\pm(z)=\exp\left(\mp\sum_{j=1}^\infty \frac{z^{\pm j}}jh_{\pm j}\right)
$$
and the operators $\Ga_\pm(z)$ satisfy the commutation relation
\beq\label{Ga}
\Ga_+(z)\Ga_-(w)=\Ga_-(w)\Ga_+(z)\left(1-\frac zw\right)^2.
\eeq

Using the
boson--fermion correspondence (see \cite[Ch.~14]{Kac}), we can
identify $\H_0$ with the space $\La\otimes\C[q,q^{-1}]$, where $\La$ is
the algebra
of symmetric functions (see \cite{Mac}). In particular, consider
the charge~0 subspace $\H[0]=\H_0[0]$, which is
identified with $\La$. We can use the following representation of the
Heisenberg algebra generated by $\{h_n\}_{n\in\Z}$:
\beq\label{power}
h_n\leftrightarrow 2n\frac\partial{\partial p_n},\qquad
h_{-n}=p_n,\qquad n>0,
\eeq
where $p_j$ are Newton's power sums. Then
the corresponding Virasoro operators are
\beq\label{vir0}
\begin{aligned}
L_n&=\sum_{r=n+1}^\infty p_{n-r}\cdot r\frac\partial{\partial p_r}+
\sum_{r=1}^{n-1}r(n-r)\cdot\frac\partial{\partial
p_r}\frac\partial{\partial p_{n-r}},\\
L_{-n}&=\sum_{r=1}^\infty p_{n+r}\cdot r\frac\partial{\partial p_r}+
\frac14\sum_{r=1}^{n-1}p_rp_{n-r},\qquad n>0.
\end{aligned}
\eeq
Note that the representation~\eqref{power} of the Heisenberg algebra,
and hence the representation~\eqref{vir0} of the Virasoro algebra, are
not unitary with respect to the standard inner product in $\La$. To
make it unitary, we should consider the inner product in $\La$ defined
by
\beq
\langle p_\la,p_\mu\rangle=\de_{\la\mu}\cdot z_\la\cdot 2^{l(\la)},
\eeq
where $p_\la$ are the power sum symmetric functions, $z_\la=\prod_i
i^{m_i}m_i!$ for a Young diagram $\la$ with $m_i$ parts of length $i$, and
$l(\la)$ is the length (number of nonzero rows) of $\la$.

Denote the singular vectors of $\Vir$ in $\H[0]$ by
$\xi_m:=\xi_{0,m}$. According to a result by Segal \cite{Segal}, in
the symmetric function realization~\eqref{power},
\beq\label{segal}
\xi_{n}\leftrightarrow c\cdot s_{(n^n)},
\eeq
where $s_{(n^n)}$ is the Schur function indexed by the $n\times n$
square Young diagram and $c$ is a numerical coefficient.

\subsection{Fusion product and the Fock space}
It is shown in \cite{FF} that
$$
V_{2n}^*\simeq\C[e_0,{\ldots} ,e_{-(2n-1)}]\Om_{-2n}\subset\F
$$
as an $\sl\otimes(\C[t^{-1}]/t^{-2n})$-module, the embedding $j_{2n}$
under this isomorphism coincides with the natural inclusion
$$\C[e_0,{\ldots} ,e_{-(2n-1)}]\Om_{-2n}\subset\C[e_0,{\ldots}
,e_{-(2n+1)}]\Om_{-2(n+1)},$$ and the limit space $\cal V$ coincides
with $\H_0$.

Using results of \cite{FF}, one can easily prove the following lemma.

\begin{lemma}\label{l:ebasis}
A basis in $F_{2n}=\C[e_0,{\ldots},e_{-(2n-1)}]\Om_{-2n}$ is
$$
\{e_0^{i_0}e_{-1}^{i_1}{\ldots} e_{-(2n-1)}^{i_{2n-1}}:0\le k\le
2n-(i_0+{\ldots} +i_{2n-1})\}\Om_{-2n}.
$$
\end{lemma}

Observe that under this ``fusion--Fock'' correspondence,
the charge 0 subspace $\H[0]$ is identified with ${\cal
V}[0]$. It follows from Lemma~\ref{l:ebasis} that
a basis of $F_{2n}[0]=F_{2n}\cap\H[0]$ is
\beq\label{ebasis0}
\{\prod e_0^{i_0}e_{-1}^{i_1}{\ldots}e_{-n}^{i_n}:i_0+i_1+{\ldots}
+i_{n}=n\}\Om_{-2n}.
\eeq

\section{The key isomorphism in more detail}\label{sec:iso}

Comparing \eqref{fpdecomp} and \eqref{virdecomp}, we obtain the
following result.

\begin{corollary}
The space $\Pi_k$ of the discrete representation of the infinite
symmetric group corresponding to the tableau $\tau_k$ has a natural
structure of the Virasoro module $L(1,k^2)$.
\end{corollary}

Our aim is to study this Virasoro representation in $\Pi_k$ (or, which
is equivalent, the corresponding representation of the infinite
symmetric group in the Fock space). In particular, from the known
theory of the basic module $L_{0,1}$, we immediately obtain the following
result.

\begin{corollary}\label{cor:vir}
In the above realization of the Virasoro module $L(1,k^2)$,
the Gelfand--Tsetlin basis in $\Pi_k$ (which consists of the infinite
two-row Young tableaux tail-equivalent to $\tau_k$) is the eigenbasis of
$L_0$, and the eigenvalues are given by the stable major index $r$:
$$
L_0\tau=r(\tau)\tau.
$$
\end{corollary}

Note that,
in view of~\eqref{corr0} and the remark after Lemma~\ref{l:ebasis},
the charge~$0$ subspace $\H[0]$ is identified with the space
$\Pi$ spanned by all infinite two-row Young tableaux with
``correct'' tail behavior. Thus we obtain the following corollary.

\begin{corollary}
The space $\Pi$, which is the countable sum of discrete
representations of the infinite symmetric group $\sinf$, has a
structure of an irreducible representation of the Heisenberg algebra
$\mathfrak A$.
\end{corollary}

On the other hand, as mentioned
above, $\H[0]$ can be identified with the algebra of symmetric
functions $\La$ via~\eqref{power}.
Denote by $\Phi$ the obtained isomorphism between $\Pi$ and $\La$, which thus associates with every tableau
$\tau\in\Pi$ a symmetric function $\Phi(\tau)\in\La$ such
that $r(\tau)=\deg\Phi(\tau)$.

Denote
by $T^{(N)}$ the (finite) set of two-row tableaux that coincide with some
$\tau_n$, $n=0,1,{\ldots} $, from the $N$th level. Let $\Pi^{(N)}$ be
the subspace in $\Pi$ spanned by all $\tau\in T^{(N)}$. It follows
from all the above identificatons that
$\Pi^{(2k)}\leftrightarrow F_{2k}[0]$.

\begin{theorem}\label{th3}
Under the isomorphism $\Phi$,
\begin{enumerate}
\item[{\rm1)}]
the principal
tableaux~{\rm\eqref{tauk}} correspond to the Schur functions with square Young
diagrams:
$$\Phi(\tau_k)={\rm const}\cdot s_{(k^k)};
$$

\item[{\rm2)}] the subspace $\Pi^{(2k)}$
correspond to the subspace $\La_{k\times k}$ of $\La$ spanned by the
Schur functions indexed by Young diagrams lying in the $k\times k$
square; the correspondence between the Schur function basis in
$\La_{k\times k}$ and the basis~\eqref{ebasis0} in $\Pi^{(2k)}\simeq
F_{2k}[0]$ is given by
formula~\eqref{gensegal} below.
\end{enumerate}
\end{theorem}

\begin{proof}

We follow Wasserman's \cite{Was} proof of Segal's
result~\eqref{segal}.

Let $0\le i_1,{\ldots} ,i_k\le k$. Then, obviously,
$$
e_{-i_1}{\ldots} e_{-i_k}\Om_{-2k}=\left[\prod_{j=1}^kz_j^{i_j-1}\right]
E(z_k){\ldots} E(z_1)\Om_{-2k},
$$
where by $[\mbox{monomial}]F(z_1,{\ldots} ,z_m)$ we denote the coefficient of this
monomial in $F(z_1,{\ldots} ,z_m)$. Now, using the
representation~\eqref{EF}, the commutation relation~\eqref{Ga}, and
the obvious facts that $V^{-k}\Om_{-2k}=\Om_0$ and
$\Ga_+(z)\Om_0=\Om_0$, we obtain
$$
E(z_k){\ldots}E(z_1)\Om_{-2k}=\prod_{j=1}^kz_j^{2(k-j)}
\prod_{1\le j<i\le k}\left(1-\frac{z_i}{z_j}\right)^2\Ga_-(z_k){\ldots}
\Ga_-(z_1)\Om_0.
$$
Observe that, in view of~\eqref{power} and the well-known fact from
the theory of symmetric functions, $\Ga_-(z)$ is
exactly the generating function of the complete symmetric functions. Hence,
expanding the product $\Ga_-(z_k){\ldots}\Ga_-(z_1)\Om_0$ by the Cauchy identity
(\cite[I.4.3]{Mac}) and making simple transformations, we obtain
$$
E(z_k){\ldots}E(z_1)\Om_{-2k}=
(-1)^{k(k-1)/2}
\prod_{j=1}^kz_j^{k-1}a_\de(z)a_\de(z^{-1})
\sum_{\la:\,l(\la)\le k} s_\la(z^{-1})s_\la,
$$
where
$$
a_\de(z)=\prod_{1\le i<j\le k}(z_i-z_j)=\det[z_i^{k-j}]_{1\le i,j,\le k}
$$
is the Vandermonde determinant, $a_\de(z^{-1})$ is the similar
determinant for the variables $z^{-1}=(z_1^{-1},{\ldots}
,z_k^{-1})$,  $l(\la)$ is the length of the diagram $\la$ (the number of
nonzero rows), $s_\la(z^{-1})$ is the Schur function
calculated at the variables $z^{-1}$, and
$s_\la$ is the Schur function as an element of $\La$ identified
with $\H[0]$. Thus we have
$$
e_{-i_1}{\ldots} e_{-i_k}\Om_{-2k}=(-1)^{k(k-1)/2}\cdot[1]\left(\prod_{j=1}^kz_j^{k-i_j}
a_\de(z)a_\de(z^{-1})
\sum_\la s_\la(z^{-1})s_\la\right).
$$

First consider the case where $i_1={\ldots}=i_k=m$. Then, by the
definition of the Schur functions \cite[I.3.1]{Mac},
$$
\prod z_j^{k-i_j}a_\de(z)=\det[z_i^{2k-m-j}]_{1\le i,j\le k}
=a_\de(z)s_{((k-m)^k)}(z),
$$
where $((k-m)^k)$ is the rectangular Young diagram with $k$ rows of
length $k-m$,
and the standard orthogonality relations imply that
\beq\label{oursegal}
e_{-m}^k\Om_{-2k}=(-1)^{k(k-1)/2}k!\cdot s_{((k-m)^k)}.
\eeq
Since $\xi_k=e_0^k\Om_{-2k}$, for $m=0$ this is Segal's result~\eqref{segal},
which we have now extended to the case of rectangular diagrams.
It is easy to see that the singular vector of $\Vir$ in ${\cal
V}_k[0]$ is just $e_0^k\om_{-2k}\otimes\tau_k$, so that the first
claim of the theorem follows.

We now turn to the case of $i_1,{\ldots} ,i_k$ that are not necessarily
equal. For convenience, set $\te_p:=e_{-(k-p)}$, $0\le p\le k$. Given
$0\le\al_1,{\ldots},
\al_k\le k$, we have
\beq\label{comp1}
\te_{\al_1}{\ldots}\te_{\al_k}\Om_{-2k}
=[1]\left(\prod_{j=1}^kz_j^{\al_j}a_\de(z)\sum_{l(\la)\le
k}a_{\la+\de}(z^{-1})s_\la\right),
\eeq
where $a_{\la+\de}(x)=\det[x_i^{\la_j+k-j}]_{1\le i,j\le k}=s_\la(x)a_\de(x)$. Consider
a Young diagram $\mu=(\mu_1,{\ldots}
,\mu_k)= (0^{r_0}1^{r_1}2^{r_2}{\ldots} )$. Let us
sum~\eqref{comp1}
over all different permutations $\al=(\al_1,{\ldots}
,\al_k)$ of the sequence
$(\mu_1,{\ldots} ,\mu_k)$. Note that
the operators $e_j$ commute with each other, so that the
left-hand side does not depend on the order of the factors. In the
right-hand side, $\sum_\al\prod z_j^{\al_j}=m_\mu(z)$, a monomial
symmetric function. Thus we have
\beq\label{comp2}
\frac{k!}{\prod_{j=0}^k r_j!}\te_{\mu_1}{\ldots} \te_{\mu_k}=
[1]\left(m_\mu(z)a_\de(z)\sum_{l(\la)\le
k}a_{\la+\de}(z^{-1})s_\la\right).
\eeq
Let $\nu$ be a Young diagram with
at most $k$ rows and at most $k$ columns, i.e., $\nu\subset(k^k)$. We have
\beq\label{Kostka}
s_\nu(z)=\sum_{\mu}
K_{\nu\mu}m_{\mu}(z),
\eeq
where $K_{\nu\mu}$ are Kostka numbers. It is well known that
$K_{\nu\mu}=0$ unless $\mu\le\nu$, where
$\le$ is the standard ordering on partitions:
$\mu\le\nu\iff \mu_1+{\ldots} +\mu_i\le\nu_1+{\ldots} +\nu_i$ for
every $i\ge1$. In particular, $\mu_1\le \nu_1\le k$.
Besides, since we consider
only $k$ nonzero variables $z_1,{\ldots} ,z_k$, it also follows that
$m_\mu(z)=0$ unless $l(\mu)\le k$. Thus the sum in~\eqref{Kostka} can
be taken only over diagrams $\mu\subset(k^k)$, for which equation~\eqref{comp2}
holds. Multiplying this equation by $K_{\nu\mu}$ and summing
over $\mu$ yields
$$
\sum_{\mu=(0^{r_0}1^{r_1}2^{r_2}{\ldots} )\subset(k^k)}
\frac{k!}{\prod_{j=0}^k r_j!}K_{\nu\mu}\te_{\mu_1}{\ldots} \te_{\mu_k}=
[1]\left(s_\nu(z)a_\de(z)\sum_{l(\la)\le
k}a_{\la+\de}(z^{-1})s_\la\right).
$$
By the orthogonality relations, the right-hand side is equal to
$k!s_\nu$. Thus we obtain the following formula:
\beq\label{gensegal}
s_\nu=\sum_{\mu=(0^{r_0}1^{r_1}2^{r_2}{\ldots} )\subset(k^k)}\frac{K_{\nu\mu}}
{\prod_{j=0}^k r_j!}e_{-(k-\mu_1)}{\ldots}
e_{-(k-\mu_k)}\Om_{-2k}.
\eeq
Observe that for rectangular diagrams this formula reduces
to~\eqref{oursegal}. Indeed, for $\nu=((k-m)^k)$, all
diagrams $\mu$ with $\mu<\nu$ have $l(\mu)>k$, hence the only nonzero
term in the right-hand side of~\eqref{gensegal} corresponds to
$\mu=\nu$, with $K_{\nu\nu}=1$ and $r_j=k!\de_{j,k-m}$.

It follows from~\eqref{gensegal} that $\La_{k\times k}\subset\Pi^{2k}$.
On the other hand,
the generating functions for the tableaux from
$T^{(2k)}$ and for the Young diagrams lying in the $k\times k$ square
coincide:
$$
\sum_{\tau\in T^{(2k)}}q^{r(\tau)}=\sum_{\la\subset (k^k)}q^{|\la|}=
\left[2k\atop k\right]_q,
$$
where $|\la|$ is the number of cells in a Young diagram $\la$ and
$\left[2k\atop k\right]_q$ is the $q$-binomial coefficient (the
equation for Young diagrams can be found in \cite[Theorem~3.1]{And};
for tableaux, it can be deduced from the known results on the major
index given, e.g., in \cite{St}). This
implies, in particular, that $\dim\La_{k\times k}=\dim\Pi^{2k}$ and completes the proof.
\end{proof}

\end{document}